\documentclass[12pt,twoside]{amsart}
\usepackage{amsmath}
\usepackage{amssymb}
\usepackage{wasysym}
\usepackage{psfrag}
\usepackage{graphicx,epsf,amsmath}  
\usepackage{epsf,graphicx}
\usepackage{comment}
  
\newtheorem{thm}{Theorem}[section]
\newtheorem{theorem}{Theorem}[section]

\newtheorem{remark}[thm]{Remark}

\newtheorem{lemma}[thm]{Lemma}

\def\RR{{\mathbb R}}

\def\C{{\mathcal C}}

\newcommand{\bEm}{\mathbb{E}_m}
\newcommand{\diver}{{\rm div}}
\newcommand{\N}{\mathbb{N}}
\newcommand{\R}{\mathbb{R}}
\newcommand{\divg}{\diver_\gamma}
\newcommand{\Dg}{{\Delta_\gamma}}
\newcommand{\gga}{\gamma}
\newcommand{\Ldeu}{L^2_\gamma(X)}

\newcommand{\FCb}{{\mathcal F C}_b^1}
\newcommand{\FCbt}{{\mathcal F C}_b^2}
\newcommand{\FCbk}{{\mathcal F C}_b^k}

\newcommand\scal[2]{{\left\langle #1 ,#2\right\rangle}}

\def\1{{\bf 1}}

\title[]{On some pointwise inequalities involving nonlocal operators}
\author{Luis A. Caffarelli}
\author{Yannick Sire}

\begin{document}
\maketitle

\begin{abstract}
The purpose of this paper is three-fold: first, we survey on several known pointwise identities involving fractional operators; second, we propose a unified way to deal with those identities; third, we prove some new pointwise identities in different frameworks in particular geometric and infinite-dimensional ones. \end{abstract}

\begin{center}
{\sl To Dick Wheeden, on the occasion of his $70$th birthday, with admiration and affection}
\end{center}

\tableofcontents

\section{Introduction}

The present paper is devoted to several pointwise inequalities involving several nonlocal operators. We focus on two types of pointwise inequalities: the C\'ordoba-C\'ordoba inequality and the Kato inequality. In order to keep the presentation simple, we state the inequalities in question in the case of the fractional laplacian, i.e. $(-\Delta)^s$,  in $\RR^n$. Actually, in subsequent sections, we will generalize these inequalities to a  lot of different contexts. Furthermore, we will present a unified proof for both inequalities based on some extension properties of some nonlocal operators. Our proofs are elementary and simplify the original arguments.

The fractional  
Laplacian can be defined in various ways, which we review now. 
It can be defined using Fourier transform by 
$$
\mathcal F((-\Delta)^s v) = \left| \xi \right|^{2s}\mathcal F(v), 
$$
for $v\in H^s(\RR^n)$. It can also
be defined through the kernel representation (see the book by Landkof \cite{landkof})
\begin{equation}\label{defPV}
(-\Delta)^{s} v(x)=C_{n,s}\ \textrm{P.V.}\int_{\mathbb{R}^n} \frac{v(x)-v(\overline x)}
{|x-\overline x|^{n+2s}}\,d\overline x,
\end{equation}
for instance for $v\in\mathcal{S}(\RR^n)$, the Schwartz space of rapidly decaying functions. Here we will only consider $s \in (0,1). $ 

The inequalities considered in the present paper are the following

\begin{thm}[C\'ordoba-C\'ordoba inequality]\label{CC}
Let $\varphi$ be a $C^2(\RR^n)$ convex function. Assume that $u$ and $\varphi(u)$ are such that $(-\Delta)^s u $ and $(-\Delta)^s \varphi(u) $ exist. Then the following holds
\begin{equation}\label{CCeq}
(-\Delta)^s\varphi(u) \leq \varphi'(u) (-\Delta)^s\,u . 
\end{equation}
\end{thm}

The next theorem is the Kato inequality. 
\begin{thm}[Kato inequality]\label{kato}
The following inequality holds in the distributional sense
\begin{equation}\label{k}
(-\Delta)^s |u| \leq \text{sgn}(u) (-\Delta)^s\,u .
\end{equation}
\end{thm}

The previous two theorems are already known: Theorem \ref{CC} is due to C\'ordoba and C\'ordoba (see \cite{CCCMP, CCPNAS}). Theorem \ref{kato} is due to Chen and V\'eron (see \cite{CV}). Both original proofs are based on the representation formula given in \eqref{defPV}. This formula holds only when the fractional laplacian is defined on $\RR^n$. The C\'ordoba-C\'ordoba inequality is a very useful result in the study of the quasi-geostrophic equation (see \cite{CCCMP}). This inequality has been generalized in several contexts in \cite{CordobaAdv} for instance or \cite{constantin}. In this line of research we propose a unified way of proving these inequalities based on some extension properties for nonlocal operators. 

\section{Some new inequalities}\label{new}

In this section, we derive by a very simple argument several inequalities at the nonlocal level, i.e. without using any extensions, which are not available in these frameworks. 
\subsection{A pointwise inequality for nonlocal operators in non-divergence form}

Nonlocal operators in non-divergence form are defined by 
$$
\mathcal I u(x)=-\int_{\RR^n} (u(x+y)+u(x-y)-2u(x))K(y)\,dy
$$
for a  kernel $K \geq 0$. Denote 
$$
\delta_y u (x)= -\Big ( u(x+y)+u(x-y)-2u(x) \Big ).
$$ 
Then, considering a $C^2$ convex function $\varphi$, one has by the fact that a convex function is above its tangent line
$$
\delta_y \varphi(u)(x)=-\Big ( \varphi(u(x+y))+\varphi(u(x-y))-2\varphi(u(x))\Big )=$$
$$-\Big (\varphi(u(x+y))-\varphi(u(x))+\varphi(u(x-y))-\varphi(u(x))\Big )
$$
$$
\leq \varphi '(u(x)) \delta_y u(x).
$$
Hence for the operator $\mathcal I$ one has also an analogue of the original C\'ordoba-C\'ordoba estimate. 

\subsection{The case of translation invariant kernels}

Consider the operator 

$$
\mathcal L u(x)=\int_{\RR^n} (u(x)-u(y))K(x-y)\,dy
$$
where $K$ is symmetric. Hence one can write 
$$
\mathcal L u(x)=\int_{\RR^n} (u(x)-u(x-h))K(h)\,dh
$$
or in other words, by a standard change of variables 
$$
\mathcal L u(x)=\frac12 \int_{\RR^n} \delta_h u (x)K(h)\,dh
$$
We start with the following lemma, which is a direct consequence of the symmetry of the kernel
\begin{lemma}
$$
\int_{\RR^n} \mathcal L u(x)=0. 
$$
\end{lemma}
The following lemma is consequence of straightforward computations
\begin{lemma}
$$
\delta_h uv(x)=u\delta_h v+v\delta_h u + 
$$
$$
(v(x+h)-v(x))(u(x+h)-u(x))+(v(x-h)-v(x))(u(x-h)-u(x)). 
$$
\end{lemma}
Hence by the two previous lemma one has the useful identity
$$
0=\int_{\RR^n} \mathcal L u^2=2\int_{\RR^n} u\mathcal L u + 2\int_{\RR^n}\int_{\RR^n} (u(x)-u(y))^2K(x-y)\,dxdy. 
$$
\subsection{Some integral operators on geometric spaces}

In this section, we describe new operators involving curvature terms. These operators appear naturally in harmonic analysis, as described below. They are of the form 
$$
\mathcal L u(x)=\int (u(x)-u(y))K(x,y)\,dy
$$
where the non-negative kernel $K$ is symmetric and has some geometric meaning. The integral sign runs either over a Lie group or over a Riemannian manifold. By exactly the same argument as in the previous section, one deduces trivially C\'ordoba-C\'ordoba estimates for these operators. We now describe these new operators. 

\subsubsection*{The case of Lie groups}
Let $G$ be a  unimodular connected Lie group endowed with the Haar measure $dx$. By ``unimodular'', we mean that the Haar measure is left and right-invariant. If we denote by $\mathcal  G$  the Lie algebra  of $G$, we consider a family 
$$\mathbb X= \left \{ X_1,...,X_k \right \}$$
of left-invariant vector fields on $G$ satisfying the H\"ormander condition, i.e. $\mathcal G$ is the Lie algebra generated by the $X_i's$. A standard metric on $G$ , called the Carnot-Caratheodory metric, is naturally associated with $\mathbb X$ and is defined as follows: let $\ell : [0,1] \to G$ be an absolutely continuous path. We say that $\ell$ is admissible if there exist measurable functions $a_1,...,a_k : [0,1] \to \mathbb C$ such that, for almost every $t \in [0,1]$, one has 
$$\ell'(t)=\sum_{i=1}^k a_i(t) X_i(\ell(t)).$$
If $\ell$  is admissible, its length is defined by 
$$|\ell |= \int_0^1\left(\sum_{i=1}^k |a_i(t)|^2 \,dt \right)^{ \frac 12 }.$$

For all $x,y \in G $, define $d(x,y)$ as  the infimum of the lengths  of all admissible paths joining $x$ to $y$ (such a curve exists by the H\"ormander condition). This distance is left-invariant. For short, we denote by $|x|$ the distance between $e$, the neutral element of the group and $x$,  so that the distance from $x$ to $y$ is equal to  $|y^{-1}x|$. 

For all $r>0$, denote by $B(x,r)$ the open ball in $G$ with respect to the Carnot-Caratheodory distance and by $V(r)$ the Haar measure of any ball. There exists $d\in \N^{\ast}$ (called the local dimension of $(G,\mathbb X)$) and $0<c<C$ such that, for all $r\in (0,1)$,
$$
cr^d\leq V(r)\leq Cr^d,
$$
see \cite{nsw}. When $r>1$, two situations may occur (see \cite{guivarch}): 
\begin{itemize}
\item Either there exist $c,C,D >0$ such that, for all $r>1$, 
$$c r^D \leq V(r) \leq C r^D$$
where $D$ is called the dimension at infinity of the group (note that, contrary to $d$, $D$ does not depend on $\mathbb X$). The group is said to have polynomial volume growth. 
\item Or  there exist $c_1,c_2,C_1,C_2 >0$ such that, for all $r>1$, 
$$c_1 e^{c_2r} \leq V(r) \leq C_1 e^{C_2r}$$
and the group is said to have exponential volume growth. 
\end{itemize} 
When $G$ has polynomial volume growth, it is plain to see that there exists $C>0$ such that, for all $r>0$,
\begin{equation} \label{homog}
V(2r)\leq CV(r),
\end{equation}
which implies that there exist $C>0$ and $\kappa>0$ such that, for all $r>0$ and all $\theta>1$,
\begin{equation} \label{homogiter}
V(\theta r)\leq C\theta^{\kappa}V(r).
\end{equation}

On a Lie group as previously described, one introduces the Kohn sub-laplacian 
$$
\Delta_G=\sum_{i=1}^k X_i^2. 
$$

On any Lie group $G$, it is natural by functional calculus to define the fractional powers $(-\Delta_G)^s$, $s \in (0,1)$ of the Kohn sub-laplacian $-\Delta_G$. It has been proved in \cite{MRS, RS} (see also \cite{SW}) that for Lie groups with polynomial volume
$$
\|(-\Delta_G)^{s/2}u\|^2_{L^2(G)} \leq C \int_{G \times G} \frac{|u(x)-u(y)|^2}{V(|y^{-1}x|)|y^{-1}x|^{2s}}\,dx \,dy. 
$$

It is therefore natural to consider the operator which is the Euler-Lagrange operator of the Dirichlet form in the R.H.S. of the previous equation given by
$$
\mathcal Lu(x)=\int_G \frac{u(x)-u(y)}{V(|y^{-1}x|)|y^{-1}x|^{2s}}\,dy. 
$$

It defines a new Gagliardo-type norm, suitably designed for Lie groups (of any volume growth). By the structure itself of this norm, one can prove as before a C\'ordoba-C\'ordoba inequality. 

\subsubsection*{The case of manifolds}

Let $M$ be a complete riemannian manifold of dimension $n$. Denote $d(x,y)$ the geodesic distance from $x$ to $y$. Similarly to the previous case it is natural to introduce the new operators, Euler-Lagrange of suitable Gagliardo norms, given by 

$$
\mathcal Lu(x)=\int_{M} \frac{u(x)-u(y)}{d(x,y)^{n+2s}}\,dy
$$

These new operators also satisfy C\'ordoba-C\'ordoba estimates (see \cite{RS} for an account in harmonic analysis where these quantities pop up). 
\section{A review of the extension property}

\subsection{The extension property in $\RR^n$}\label{extRn}

We first introduce the spaces
$$H^s(\RR^n)=\left \{ v \in L^2(\RR^n)\,\,:\,\,|\xi|^{s} (\mathcal F v)(\xi) \in L^2(\RR^n) 
\right \},$$
where $s \in (0,1)$ and $\mathcal F$ denotes Fourier transform.
For $\Omega \subset \RR^{n+1}_+$ a Lipschitz domain (bounded or unbounded) and $a \in (-1,1)$,
we denote
$$H^1(\Omega,y^a)=\left \{ u \in L^2
(\Omega,y^a \,dx\,dy)\,\,:\,\,|\nabla u| \in L^2(\Omega,y^a \,dx\,dy) \right \}.$$

Let $a=1-2s$. It is well known that the space $H^s(\RR^n)$ 
coincides with the trace on $\partial\RR^{n+1}_{+}$ of $H^1(\RR^{n+1}_+,y^a)$.
In particular, every $v\in H^s(\RR^n)$ is the trace of a function $u\in L^2_{\rm loc}(\RR^{n+1}_+,y^a)$ 
such that $\nabla u \in L^2(\RR^{n+1}_+,y^a)$.  
In addition, the function $u$ which minimizes
\begin{equation} \label{argmin} 
{\rm min}\left\{ \int_{\RR^{n+1}_{+}}y^a \left| \nabla u \right|^2\;dx dy \; : \;  
u|_{\partial\RR^{n+1}_{+}}=v\right\}
\end{equation} 
solves the Dirichlet problem
\begin{equation}\label{bdyFrac2} 
\left \{
\begin{aligned} 
L_a u:= \textrm{div\,} (y^a \nabla u)&=0 \qquad 
{\mbox{ in $\RR^{n+1}_+$}} 
\\
u&= v  
\qquad{\mbox{ on $\partial\RR^{n+1}_+$.}}\end{aligned}\right . 
\end{equation} 
By standard elliptic regularity, $u$ is smooth in $\RR^{n+1}_{+}$. It turns out that 
$-y^a u_{y} (\cdot,y)$ converges in $H^{-s}(\RR^n)$ to a distribution 
$h\in H^{-s}(\RR^n)$ as $y\downarrow 0$. That is, $u$ weakly solves
\begin{equation}\label{bdyFrac3} 
\left \{
\begin{aligned} 
\textrm{div\,} (y^a \nabla u)&=0 \qquad 
{\mbox{ in $\RR^{n+1}_+ $}} 
\\
-y^a \partial_y u &= h
\qquad{\mbox{ on $\partial\RR^{n+1}_+$.}}\end{aligned}\right . 
\end{equation}
Consider the Dirichlet to Neumann operator 
$$
\begin{aligned}
&\Gamma_a: H^s(\RR^n)\to H^{-s}(\RR^n)\\
&\qquad \quad v\mapsto \Gamma_{a}(v)=  h:=
 - \displaystyle{\lim_{y \rightarrow 0^+}} y^{a} \partial_y u=\frac{\partial u}{\partial \nu^a}, 
\end{aligned}
$$
where $u$ is the solution of \eqref{bdyFrac2}. 
Then, we have:

\begin{theorem}[\cite{CS}]\label{realization}
For every $v\in H^s(\RR^n)$,
\begin{equation*}
(-\Delta)^s v= d_s\Gamma_{a}(v)= 
- d_s \displaystyle{\lim_{y \rightarrow 0^+}} y^{a} \partial_y u,
\end{equation*}
where $a=1-2s$, $d_s$ is a positive constant depending only on~$s$,
and the equality holds in the distributional sense.
\end{theorem}  

\subsection{The extension property in bounded domains}
We consider now the case of bounded domains. In this case, two different operators can be defined.

\noindent $\bullet$ {\sl The spectral Laplacian: } If one considers the classical Dirichlet Laplacian $\Delta_{\Omega}$ on the domain $\Omega$\,, then the {spectral definition} of the fractional power of $\Delta_{\Omega}$ relies on the following formulas:
\begin{equation}\label{sLapl.Omega.Spectral}
\displaystyle(-\Delta_{\Omega})^{s}
g(x)=\sum_{j=1}^{\infty}\lambda_j^s\, \hat{g}_j\, \phi_j(x)
=\frac1{\Gamma(-s)}\int_0^\infty
\left(e^{t\Delta_{\Omega}}g(x)-g(x)\right)\frac{dt}{t^{1+s}}.
\end{equation}
Here $\lambda_j>0$, $j=1,2,\ldots$ are the eigenvalues of the Dirichlet Laplacian on $\Omega$ with zero boundary conditions\,, written in increasing order and repeated according to their multiplicity and $\phi_j$ are the corresponding normalized eigenfunctions, namely
\[
\hat{g}_j=\int_\Omega g(x)\phi_j(x)\, dx\,,\qquad\mbox{with}\qquad \|\phi_j\|_{L^2(\Omega)}=1\,.
\]
The first part of the formula is therefore an interpolation definition. The second part gives an equivalent definition in terms of the semigroup associated to the Laplacian.
We will denote the operator defined in such a way as $\mathcal A_{1,s}=(-\Delta_{\Omega})^s$\,, and call it the \textit{spectral fractional Laplacian}. 

\medskip

\noindent $\bullet$~{\sl The restricted fractional laplacian:} On the other hand, one can define a fractional Laplacian operator by using the integral representation in terms of hypersingular kernels already mentioned
\begin{equation}\label{sLapl.Rd.Kernel}
(-\Delta_{\RR^d})^{s}  g(x)= C_{d,s}\mbox{
P.V.}\int_{\mathbb{R}^n} \frac{g(x)-g(z)}{|x-z|^{n+2s}}\,dz,
\end{equation}
 In this case we materialize the zero Dirichlet condition  by restricting the operator to act only on functions that are zero outside $\Omega$. We will call  the operator defined in such a way the \textit{restricted fractional Laplacian} and use the specific notation $\mathcal A_{2,s}=(-\Delta_{|\Omega})^s$ when needed. As defined,  $\mathcal A_{2,s}$ is a self-adjoint operator on $L^2(\Omega)$\,, with a discrete spectrum: we will denote by $\lambda_{s, j}>0$, $j=1,2,\ldots$ its eigenvalues written in increasing order and repeated according to their multiplicity and we will denote by $\{\phi_{s, j}\}_j$ the corresponding set of eigenfunctions, normalized in $L^2(\Omega)$. 

\medskip

\noindent $\bullet$~{\sl Common notation.}
In the sequel we use $\mathcal A$ to refer to any of the two types of operators $\mathcal A_{1,s}$ or $\mathcal A_{2,s}$, $0<s<1$. Each one is defined on a Hilbert space
\begin{equation} \label{defH}
H(\Omega)=\{u=\sum_{k=1}^\infty u_k \phi_{s,k} \in L^2(\Omega)\; : \; \| u \|^2_{H} = \sum_{k=1}^\infty \lambda_{s,k}\vert u_k\vert^2 <+\infty\}\subset L^2(\Omega)
\end{equation}
with values in its dual $H^*$. The notation in the formula copies the one just used for the second operator. When applied to the first one we put here  $\phi_{s,k}=\phi_k$, and $\lambda_{s,k}=\lambda_{k}^s$.
 Note that $H(\Omega)$ depends in principle on the type of operator and on the exponent $s$.  Moreover, the operator $\mathcal A$ is an isomorphism between $H$ and $H^*$, given by its action on the eigen-functions. 
It has been proved in \cite{BSV} (see also \cite{CDDS}) that
$$
H(\Omega)=
\left\{
\begin{aligned}
H^s(\Omega)&\qquad\text{if $s\in (0,1/2)$},\\
H^{1/2}_{00}(\Omega)&\qquad\text{if $s=1/2$},\\
H^s_{0}(\Omega)&\qquad\text{if $s\in (1/2,1)$},
\end{aligned}
\right.
$$

We now introdruce the Caffarelli-Silvestre extension for these operators. In the case of the restricted fractional laplacian, the extension is precisely the one described in Section \ref{extRn}. We now concentrate on the case of the spectral fractional laplacian. Let us define
\begin{align*}
\C &= \Omega\times(0,+\infty),\\
\partial_L \C&=\partial \Omega \times [0,+\infty ).
\end{align*}
We write points in the cylinder using the notation $(x,y)\in \C=\Omega \times (0,+\infty)$. Given $s\in(0,1)$, it has been proved in \cite{CDDS} (see also \cite{CabreTan}) that the following holds.

\begin{lemma}\label{extCylindrical}
Consider a weak solution of
\begin{equation}
\left\{
\begin{array}{lll}
\mbox{\rm div}(y^{1-2s} \nabla w)=0 &\mbox{in }\,\,\mathcal C= \Omega \times (0,+\infty),\\
w=0\,,\; &\mbox{on }\,\,\partial \Omega \times (0,+\infty) \\
\end{array}
\right.
\end{equation}
Then $-\lim_{y \to 0} y^{1-2s}\partial_y w=\mathcal A w(\cdot,0).$
where $\mathcal A$ is the spectral fractional laplacian.
\end{lemma}

\subsection{The extension property in general frameworks}\label{extST}

To generalize the inequalities under consideration, one has to invoke a rather general version of the Caffarelli-Silvestre extension proved by Stinga and Torrea \cite{ST}. Their approach, based on semi-group theory, allows to prove the previous results in quite general ambient spaces, like Riemannian manifolds or Lie groups. 

In the following theorem, we will consider three cases later for the object $\mathcal M$:
\begin{enumerate}
\item The case of complete Riemannian manifolds and the Laplace-Beltrami operator
\item The case of Lie groups and the Kohn laplacian 
\item The case of the Wiener space and the Ornstein-Uhlenbeck operator 
\end{enumerate}

Let $\mathcal L$ be a positive and self-adjoint operator in 
$L^2(\mathcal M)$.  One can define its fractional powers 
by means of the standard formula in spectral theory  
\[
\displaystyle \mathcal L^{s}=\frac1{\Gamma(-s)}\int_0^\infty
\left(e^{t\mathcal L}-\mbox{Id}\right)\frac{dt}{t^{1+s}},
\]
where $s\in (0,1)$
and $e^{t\mathcal L}$ denotes the heat semi-group on $\mathcal M$. Then one has

\begin{theorem}\label{extension}
Let $u \in \text{dom}(\mathcal L^s)$. A solution of the extension problem 
\[
\left \{ \begin{array}{ll}
\displaystyle{\mathcal L v+ \frac{1-2s}{y}\partial_yv + \partial^2_yv=0}\qquad &\mbox{on } \, \mathcal M\times \R^+
\\ \\ 
v(x,0)=u &\mbox{on } \, \mathcal M, \\
\end{array} \right . 
\]
is given by 
\[
v(x,y)=\frac{1}{\Gamma(s)}\int_0^\infty e^{t\mathcal L} (\mathcal L^su)(x)e^{-y^2/4t}\frac{dt}{t^{1-s}}
\]
and furthermore, one has at least in the distributional sense 
\begin{equation}\label{eqcs}
-\lim_{y\to 0^+} y^{1-2s}\partial_yv(x,y)=
\frac{2s\Gamma(-s)}{4^s\Gamma(s)}\mathcal L^s u(x). 
\end{equation}
\end{theorem}

\section{Proofs of Theorem \ref{CC} and \ref{kato}}

\subsection{Proof of Theorem \ref{CC}}
We now come to the proof of Theorem \ref{CC}. We introduce the function 
$$
\tilde w = \varphi(w)-v
$$
where $w$ is the Caffarelli-Silvestre extension of $u$ and $v$ the Caffarelli-Silvestre extension of $\varphi(u)$. Then $\tilde w$ satisfies 
\begin{equation*}
\left \{
\begin{array}{c}
L_a \tilde w=y^a \varphi '' (w) |\nabla w|^2 \geq 0, \,\,\,\,\,\mbox{in}\,\,\,\RR^{n+1}_+\\
\tilde w=0\,\,\,\,\,\mbox{on}\,\,\,\partial \RR^{n+1}_+
\end{array} \right . 
\end{equation*}
since $\varphi $ is convex. Hence by the Hopf lemma in \cite{CS1} (see also the Appendix) ( notice $\tilde w \geq 0$ by the weak maximum principle) , one has $\frac{\partial \tilde w}{\partial \nu^a} > 0$, hence the result.

\subsection{Proof of Theorem \ref{kato}}

We now turn to the proof of the Kato inequality in Theorem \ref{kato}. This is a consequence of the Cordoba-Cordoba inequality. Indeed consider the convex function 
$$
\varphi_\epsilon(x)=\sqrt{x^2+\epsilon^2}. 
$$
Then the result follows by Theorem \ref{CC} and a standard approximation argument. 
\subsection{The results in bounded domains}

In the case of the spectral laplacian, the C\'ordoba-C\'ordoba estimate has been proved by Constantin and Ignatova \cite{constantin} by a rather involved use of semi-group theory. Our proof has the same flavour as the one of Theorem \ref{CC}. Furthermore, in our framework, one can also prove the C\'ordoba-C\'ordoba estimate in the case of the restricted laplacian, which is not covered by \cite{constantin}. 

\begin{thm}\label{CCbdd}
Let $\varphi$ be a $C^2(\RR^n)$ convex function. Assume that $u$ and $\varphi(u)$ are such that $\mathcal A u $ and $\mathcal A \varphi(u) $ exist where $\mathcal A$ is either the restricted or spectral fractional laplacian. Then the following holds
\begin{equation}\label{CCeq}
\mathcal A \varphi(u) \leq \varphi'(u) \mathcal A \,u 
\end{equation}
\end{thm}
\begin{proof}
The case of the retricted laplacian is fully covered by the proof of Theorem \ref{CC} verbatim. In the case of the spectral fractional laplacian, one considers as before 
$$
\tilde w = \varphi(w)-v
$$
where $w$ is the Caffarelli-Silvestre extension of $u$ and $v$ the Caffarelli-Silvestre extension of $\varphi(u)$ where the Caffarelli-Silvestre extension is the one described in Section \ref{extCylindrical}. Then $\tilde w$ satisfies 
\begin{equation*}
\left \{
\begin{array}{c}
L_a \tilde w=y^a \varphi '' (w) |\nabla w|^2 \geq 0, \,\,\,\,\,\mbox{in}\,\,\,\C\\
\tilde w=0\,\,\,\,\,\mbox{on}\,\,\,\partial_L \C \\
\tilde w=0\,\,\,\,\,\mbox{on}\,\,\,\left \{ y=0 \right \}
\end{array} \right . 
\end{equation*}
By the weak maximum principle, one has $\tilde w\geq 0$ in $\C$ and one concludes with the Hopf lemma in the appendix.  
\end{proof}

\begin{remark}
Our proof of the estimate is the same as the one in  C\'ordoba and Mart\'inez in \cite{CordobaAdv} for the Dirichlet-to-Neumann operator. However, their proof covers only the case $1/2$ and for power-like convex functions. The argument can be actually generalized as we mentioned. Furthermore, it unifies all the possible proofs of the C\'ordoba-C\'ordoba estimates. 
\end{remark}

\section{Geometric ambiebent spaces}

\subsection{The case of manifolds}

 The case of compact manifolds, through a parabolic argument, has been proved by Cordoba and Mart\'inez \cite{CordobaAdv}. Our proof once again completely unifies the several approaches. Consider a complete Riemannian manifold $\mathcal M$ and its Laplace-Beltrami operator
 $$
 \mathcal L=-\Delta_g
 $$
 Invoking now the extension of Stinga and Torrea described in Section \ref{extST}, one proves
 
 \begin{thm}\label{CCA}
Let $\varphi$ be a $C^2(\RR^n)$ convex function. Assume that $u$ and $\varphi(u)$ are such that $\mathcal L u $ and $\mathcal L \varphi(u) $ exist. Then the following holds
\begin{equation}\label{CCeq}
\mathcal L \varphi(u) \leq \varphi'(u) \mathcal L \,u 
\end{equation}
\end{thm}

We then recover the case of compact manifolds in \cite{CordobaAdv} and even generalize it to complete non-compact manifolds. The proof of the previous theorem is identical, once the extension is well defined as described above (see \cite{ST}), to the proof of Theorem \ref{CC}. 

\subsection{The case of Lie groups}

Consider a Lie group $G$ with its Kohn Laplacian
$$
 \mathcal L=-\Delta_G
 $$
 
  Invoking now the extension of Stinga and Torrea described in Section \ref{extST}, one proves
 
 \begin{thm}\label{CCB}
Let $\varphi$ be a $C^2(\RR^n)$ convex function. Assume that $u$ and $\varphi(u)$ are such that $\mathcal L u $ and $\mathcal L \varphi(u) $ exist. Then the following holds
\begin{equation}\label{CCeq}
\mathcal L \varphi(u) \leq \varphi'(u) \mathcal L \,u 
\end{equation}
\end{thm}

\subsection{The case of the Wiener space}
We start by recalling the basic notions about the Wiener space and its associated operators. 
An abstract Wiener space is defined as a triple $(X,\gamma,H)$ where
$X$ is a separable Banach space, endowed with the norm $\|\cdot\|_X$,
$\gamma$ is a nondegenerate centred Gaussian measure,
and $H$ is the Cameron--Martin space associated with the measure $\gamma$, that is,
$H$ is a separable Hilbert space densely embedded in $X$, endowed with the inner product
$[\cdot, \cdot ]_H$ and with the norm $|\cdot |_H$. The
requirement that $\gamma$ is a centred Gaussian measure means that
for any $x^*\in X^*$, the measure $x^*_\#\gamma$ is a centred Gaussian
measure on the real line $\RR$, that is, the Fourier transform
of $\gamma$ is given by
\[
\hat \gamma(x^*) = \int_X 
e^{-i\scal{x}{x^*}}\, d\gamma (x)=\exp\left(
-\frac{\scal{Qx^*}{x^*}}{2}
\right),\qquad \forall x^*\in X^*;
\]
here the operator $Q\in {\mathcal L}(X^*,X)$ is the covariance operator and it is
uniquely determined by the formula
\[
\scal{Qx^*}{y^*}=\int_X \scal{x}{x^*}\scal{x}{y^*}d\gamma(x),\qquad \forall x^*,y^*\in X^*.
\]
The nondegeneracy of $\gamma$ implies that $Q$ is positive definite: the boundedness
of $Q$ follows by Fernique's Theorem, 
asserting that there exists a positive number $\beta>0$
such that
\[
 \int_X e^{\beta\|x\|^2}d\gamma(x)<+\infty.
\]
This implies also that the maps $x\mapsto \scal{x}{x^*}$ belong to 
$L^p_\gamma(X)$ for any $x^*\in X^*$ and $p\in [1,+\infty)$, where $L^p_\gamma(X)$ 
denotes the space of all $\gamma$-measurable functions $f:X\to \R$
such that 
\[
\int_X |f(x)|^p d\gamma(x)<+\infty.
\]
In particular, any element $x^*\in X^*$ can be seen as a map $x^*\in L^2_\gamma(X)$, and
we denote by $R^*: X^*\to {\mathcal H}$ the identification map $R^*x^*(x):=\scal{x}{x^*}$.
The space ${\mathcal H}$ given by the closure of $R^*X^*$ in $L^2_\gamma(X)$
is usually called reproducing kernel. 
By considering the map $R: {\mathcal H}\to X$
defined as
\[
R\hat{h} := \int_X \hat{h}(x)x\, d\gamma(x),
\]
we obtain that $R$ is an injective $\gamma$--Radonifying operator, which is
Hilbert--Schmidt when $X$ is Hilbert. We also have  
$Q=RR^*:X^*\to X$. The space $H:=R{\mathcal H}$, equipped with the inner product $[\cdot,\cdot]_H$
and norm $|\cdot|_H$ induced by ${\mathcal H}$ via $R$, is the Cameron-Martin space  
and is a dense subspace of $X$. The continuity of $R$ implies
that the embedding of $H$ in $X$ is continuous, that is, there exists $c>0$
such that
\[
 \|h\|_X \leq c|h|_H,\qquad \forall h\in H.
\] 
We have also that the measure $\gamma$ is absolutely continuous with respect
to translation along Cameron--Martin directions; in fact, for $h\in H$, $h=Qx^*$, 
the measure
$\gamma_h(B)=\gamma(B-h)$ is absolutely continuous with respect to $\gamma$ with
density given by
\[
d\gamma_h(x)=\exp\left(
\scal{x}{x^*}-\frac{1}{2}|h|_H^2
\right)d\gamma(x) .
\]

For $j\in \N$ we choose $x^*_j\in X^*$ in such a way that $\hat h_j:= R^*x_j^*$, or equivalently 
$h_j:=R\hat h_j=Qx^*_j$, form an orthonormal basis of $H$. We order the vectors $x^*_j$ in such 
a way that the numbers $\lambda_j:=\|x^*_j\|_{X^*}^{-2}$ form a non-increasing sequence.
Given $m\in\mathbb N$, we also let $H_m:=\langle h_1,\ldots, h_m\rangle\subseteq H$, 
and $\Pi_m: X\to H_m$ be the closure of the orthogonal projection from $H$ to $H_m$
\[
\Pi_m(x) := \sum_{j=1}^m \scal{x}{x^*_j}\, h_j \qquad x\in X.
\]
The map $\Pi_m$ induces the decomposition $X\simeq H_m\oplus X_m^\perp$, with $X_m^\perp:= {\rm ker}(\Pi_m)$,
and $\gamma=\gamma_m\otimes\gamma_m^\perp$,
with $\gamma_m$ and $\gamma_m^\perp$ Gaussian measures on $H_m$ and $X_m^\perp$ respectively, 
having $H_m$ and $H_m^\perp$ as Cameron--Martin spaces. 
When no confusion is possible we identify $H_m$ with $\RR^m$;
with this identification the measure
$\gamma_m={\Pi_m}_\#\gamma$ is the standard Gaussian measure on $\RR^m$ (see \cite{B}).
Given $x\in X$,
we denote by $\underline x_m\in H_m$ the projection $\Pi_m(x)$, 
and by $\overline x_m\in X_m^\perp$ the infinite dimensional component of $x$, so that 
$x=\underline x_m+\overline x_m$.
When we identify $H_m$ with $\RR^m$ we rather write 
$x=(\underline x_m,\overline x_m)\in \RR^m\times X_m^\perp$.

We say that $u:X\to \RR$ is a {\em cylindrical function} if $u(x)=v(\Pi_m (x))$ for some $m\in\mathbb N$ and 
$v:\RR^m\to \RR$. 
We denote by $\FCbk(X)$, $k\in\mathbb N$, 
the space of all $C^k_b$ cylindrical functions, that is, functions of the form $v(\Pi_m (x))$
with $v\in C^k(\RR^n)$, with continuous and bounded derivatives up to the order $k$. 
We denote by $\FCbk(X,H)$ the space generated by all functions of the form 
$u h$, with $u\in \FCbk(X)$ and $h\in H$.

Given $u\in \Ldeu$, we consider the canonical cylindrical approximation $\bEm$ given by
\begin{equation}\label{cancylapprox}
\bEm u (x)=\int_{X_m^\perp} u(\Pi_m(x),y) \,d\gga_m^\perp(y).
\end{equation}
Notice that $\bEm u$ depends only on the first $m$ variables 
and $\bEm u$ converges  to $u$ in $L^p_\gamma(X)$ for all $1\leq p<\infty$.

We let
\[
\begin{array}{ll}
\displaystyle{\nabla_\gamma u := \sum_{j\in\mathbb N}\partial_j u\, h_j} & {\rm for\ } u\in \FCb(X)
\\
\\
\displaystyle{\divg \varphi := \sum_{j\geq 1}\partial^*_j [\varphi,h_j]_H} & {\rm for\ }\varphi\in \FCb(X,H)
\\
\\
\displaystyle{\Dg u := \divg\nabla_\gamma u} & {\rm for\ } u\in \FCbt(X)
\end{array}
\]
where $\partial_j := \partial_{h_j}$ and
$\partial_j^* := \partial_j - \hat h_j$ is the adjoint operator of $\partial_j$.
With this notation, the following integration by parts formula holds:
\begin{equation}\label{inp}
\int_X u\, \divg \varphi\,d\gamma = -\int_X [\nabla_\gamma u,\varphi]_H\, d\gamma
\qquad \forall \varphi\in \FCb(X,H).
\end{equation}
In particular, thanks to \eqref{inp}, the operator $\nabla_\gamma$ is closable in $L^p_\gamma(X)$,
and we denote by $W^{1,p}_\gamma(X)$ the domain of its closure. The Sobolev spaces 
$W^{k,p}_\gamma(X)$, with $k\in\mathbb N$ and $p\in [1,+\infty]$, can be defined analogously \cite{B},
and $\FCbk(X)$ is dense in $W^{j,p}_\gamma(X)$, for all $p<+\infty$ and $k,j\in\mathbb N$ with $k\ge j$.

Given a vector field $\varphi \in L^{p}_\gamma(X;H)$, $p\in (1,\infty]$, using \eqref{inp} we can define
$\mathrm{div}_\gamma \, \varphi$ in the distributional sense,
taking test functions $u$ in  $W^{1,q}_\gamma(X)$ with
$\frac{1}{p}+\frac{1}{q} = 1$. We say that
$\mathrm{div}_\gamma\, \varphi \in L^p_\gamma(X)$ if this linear functional can be extended to all test 
functions $u\in L^{q}_\gamma(X)$. This is true in particular if $\varphi\in W^{1,p}_\gamma(X;H)$.

Let $u\in W^{2,2}_\gamma(X)$, $\psi\in \FCb(X)$ and $i,j\in \mathbb N$.
{}From \eqref{inp}, with $u=\partial_j u$ and $\varphi=\psi h_i$, we get
\begin{equation}\label{parts}
\int_X \partial_j u\,\partial_{i}\psi \,d\gamma = 
\int_X -\partial_i(\partial_{j}u)\,\psi+ \partial_ju\,\psi\langle x,x^*_i\rangle d\gamma
\end{equation}
Let now $\varphi\in \FCb(X,H)$. If we apply \eqref{parts} with $\psi=[\varphi,h_j]_H=:\varphi^j$, we obtain 
\[
\int_X \partial_j u\,\partial_{i}\varphi^j \,d\gamma = 
\int_X -\partial_j(\partial_{i}u)\,\varphi^j
+ \partial_ju\,\varphi^j\langle x,x^*_i\rangle d\gamma
\]
which, summing up in $j$, gives 
\[
\int_X [\nabla_\gamma u,\partial_i \varphi]_H\,d\gamma = \int_X -[\nabla_\gamma (\partial _i u), \varphi]_H 
+  [\nabla_\gamma u,\varphi]_H \langle x,x^*_i\rangle d\gamma
\]
for all $\varphi\in \FCb(X,H)$.

The operator $\Dg:W^{2,p}_\gamma(X)\to L^p_\gamma(X)$ is usually called the 
Ornstein-Uhlenbeck operator on $X$.
Notice that, if $u$ is a cylindrical function, that is $u(x)=v(y)$ with  
$y=\Pi_m(x)\in\R^m$ and $m\in\mathbb N$,
then
\[
\Dg u = \sum_{j=1}^m \partial_{jj}u-\langle x,x_j^*\rangle\partial_{j}u = 
\Delta v - \langle \nabla v,y \rangle_{\R^m}\,.
\]
We write $u\in C(X)$ if $u: X\to \R$ is continuous and $u\in C^1(X)$ if
both $u: X\to \R$ and $\nabla_\gamma u:X\to H$ are continuous. 

{F}or simplicity of notation, from now on we omit the explicit dependence on $\gamma$ 
of operators and spaces.
We also indicate by $[\cdot, \cdot ]$ and $|\cdot |$ respectively the inner product and the norm in $H$.

By means of Section \ref{extST}, one can prove an extension property for the operator $(-\Delta_\gamma)^s$ and one proves in this case also a C\'ordoba-C\'ordoba estimate.

\section{Appendix}
In this appendix, we provide the Hopf lemma, which is crucial in the proof of the estimates. We state the theorem in the case of $\RR^n$ as stated in \cite{CS}. However, an inspection of the proof shows that it is extendable to cylinders $\mathcal M \times (0,+\infty)$ where $\mathcal M$ is one of the cases covered in the present note and the associated operators. Indeed, the geometry is always the same and the Hopf lemma just depends on the structure of the equation.

We start with some notations. We introduce
\begin{align*}
& B_R^+=\{ (x,y)\in\R^{n+1} : y>0, |(x,y)|<R\}, \\
& \Gamma_R^0=\{ (x,0)\in\partial\R^{n+1}_+ : |x|<R\}, \\
& \Gamma_R^+=\{ (x,y)\in\R^{n+1} : y\ge 0, |(x,y)|=R\}. 
\end{align*}
\begin{lemma}\label{hopf}
Consider the cylinder 
$\mathcal C_{R,1}= \Gamma_R^0 \times (0,1) \subset \RR^{n+1}_+$ where $\Gamma_R^0$ 
is the ball of center $0$ and radius $R$ in $\RR^n$. Let $u \in C(\overline{\mathcal C_{R,1}}) 
\cap H^1(\mathcal C_{R,1},y^a)$ satisfy  
\begin{equation*}
\begin{cases}
L_a u \leq 0&\text{ in } \mathcal C_{R,1} \\
u> 0&\text{ in } \mathcal C_{R,1} \\ 
u(0,0)=0.&
\end{cases}
\end{equation*}

Then, 
$$\limsup_{y \rightarrow 0^+ }-y^a \frac{u(0,y)}{y}<0.$$
In addition, if $y^a u_y \in C(\overline{\mathcal C_{R,1}})$, then 
$$\partial_{\nu^a} u (0,0) <0. $$ 
\end{lemma}

\section*{Acknowledgements}
Luis Caffarelli is supported by NSF grant DMS-1160802.

\bibliographystyle{alpha}
\bibliography{biblio}

\bigskip\bigskip\bigskip\bigskip 

LC- University of Texas at Austin, Mathematics Department 
2515 Speedway Stop C1200
Austin, Texas 78712-1202, caffarel@math.utexas.edu 

\bigskip\bigskip

YS- Johns Hopkins University, Department of Mathematics, Krieger Hall, 
3400 N. Charles Street, Baltimore, MD 21218, sire@math.jhu.edu
 
\end{document}